\documentclass[12pt, reqno]{amsart}
\usepackage{amsmath, amsthm, amscd, amsfonts, amssymb, graphicx}
\usepackage[bookmarksnumbered,plainpages]{hyperref}
\input{mathrsfs.sty}

\newtheorem{theorem}{Theorem}[section]

\newtheorem{corollary}[theorem]{Corollary}
\theoremstyle{definition}

\newtheorem{example}[theorem]{Example}

\theoremstyle{remark}
\newtheorem{remark}[theorem]{Remark}

\numberwithin{equation}{section}

\begin{document}

\title{Operator extensions of Hua's inequality}
\author[M.S. Moslehian]{M. S. Moslehian}
\address{Department of Pure Mathematics, Ferdowsi University of Mashhad,
P.O. Box 1159, Mashhad 91775, Iran; \newline Center of Excellence in
Analysis on Algebraic Structures (CEAAS), Ferdowsi University of
Mashhad, Iran.} \email{moslehian@ferdowsi.um.ac.ir and
moslehian@ams.org}



\begin{abstract}
We give an extension of Hua's inequality in pre-Hilbert
$C^*$-modules without using convexity or the classical Hua's
inequality. As a consequence, some known and new generalizations of
this inequality are deduced. Providing a Jensen inequality in the
content of Hilbert $C^*$-modules, another extension of Hua's
inequality is obtained. We also present an operator Hua's
inequality, which is equivalent to operator convexity of given
continuous real function.
\end{abstract}
\maketitle

\section{Introduction}

In his famous monograph ``Additive theory of prime numbers'',
Lo-Keng~Hua \cite{HUA} introduced the important inequality
$$\left(\delta- \sum^n_{i=1} x_i\right)^2+ \alpha \sum^n_{i= 1}
x^2_i\geq \frac{\alpha}{n+ \alpha}\delta^2\,,$$ where $\delta$,
$\alpha$ are positive numbers, and $x_i$ $(i= 1,2,\dots, n)$ are
real numbers. The equality holds if and only if
$x_i=\delta/(n+\alpha)$.

This result was generalized by C.L.~Wang \cite{WAN} by showing that
$$\left(\delta- \sum^n_{i=1} x_i\right)^p+ \alpha^{p-1} \sum^n_{i= 1}
x^p_i\geq \left(\frac{\alpha}{n+ \alpha}\right)^{p-1}\delta^p\,,$$
in which $\delta>0$, $\alpha>0$, $p\geq 1$ and $(x_1, \cdots, x_n)$
is a finite sequence of nonnegative real numbers with
$\sum_{i=1}^nx_i \leq \delta$, and that the sign of inequality is
reversed for $0<p<1$. In \cite{Y-H}, G.-S.~Yang and B.-K.~Han
extended this result for a finite sequence of complex numbers.
C.E.M.~Pearce and J.E.~Pe\v cari\'c \cite{P-P} generalized Hua's
inequality for real convex functions; see also \cite{DRA}.
S.S.~Dragomir and G.-S.~Yang \cite{D-Y} extended Hua's inequality in
the setting of real inner product spaces by applying Hua's
inequality for $n=1$. Their result was generalized by J.E.~Pe\v
cari\'c \cite{PEC}. There are other interpretation of Hua's
inequality; cf. \cite{T-M-K-T} and references therein.

\noindent An operator version of Hua's inequality was given by
R.~Drnov\v sek \cite{DRN}. Moreover, S.~Radas and T.~\v Siki\'c
\cite{R-S} generalized the Hua inequality for linear operators in
real inner product spaces. Now we consider certain extensions and
improvements of the above results in the setting of Hilbert
$C^*$-modules and operators on Hilbert spaces. Providing a Jensen
inequality in the content of Hilbert $C^*$-modules, another
extension of Hua's inequality is obtained. We also present an
operator Hua's inequality, which is equivalent to operator convexity
of given continuous real function.

\section{Preliminaries}
The notion of Hilbert $C^*$-module is a generalization of the notion
of Hilbert space. Let ${\mathscr A}$ be a $C^*$-algebra and
${\mathscr X}$ be a complex linear space, which is a right
${\mathscr A}$-module satisfying $\lambda(xa)=x(\lambda a)=(\lambda
x)a$ for all $x \in {\mathscr X},a \in {\mathscr A}, \lambda \in
{\mathbb C}$. The space ${\mathscr X}$ is called a \emph{ (right)
pre-Hilbert $C^*$-module over ${\mathscr A}$} if there exists an
${\mathscr A}$-inner product $\langle .,.\rangle :{\mathscr X}
\times {\mathscr X}\to {\mathscr A}$ satisfying

(i) $\langle x,x\rangle\geq 0$ (i.e. $\langle x,x\rangle$ is a
positive element of ${\mathscr A}$) and $\langle x,x\rangle=0$~~~ if
and only if~~~ $x=0$;

(ii) $\langle x, \lambda y + z\rangle=\lambda \langle x,y\rangle+
\langle x,z\rangle$;

(iii) $\langle x,ya\rangle=\langle x,y\rangle a$;

(iv) $\langle x,y\rangle^*=\langle y,x\rangle$;

\noindent for all $x, y, z \in {\mathscr X},\, \lambda \in {\mathbb
C},\, a \in {\mathscr A}$.

We can define a norm on ${\mathscr X}$ by $\| x \| :=\| \langle
x,x\rangle\| ^\frac{1}{2}$, where the latter norm denotes that in
the $C^*$-algebra ${\mathscr A}$. A pre-Hilbert ${\mathscr
A}$-module is called a \emph{ (right) Hilbert $C^*$-module over
${\mathscr A}$} (or a \emph{(right) Hilbert ${\mathscr A}$-module})
if it is complete with respect to its norm. Any inner product space
can be regarded as a pre-Hilbert $\mathbb{C}$-module and any
$C^*$-algebra ${\mathscr A}$ is a Hilbert $C^*$-module over itself
via $\langle a, b\rangle = a^*b\,\,(a, b \in {\mathscr A})$.

\noindent A mapping $T:{\mathscr X}\to {\mathscr Y}$ between Hilbert
${\mathscr A}$-modules is called adjointable if there exists a
mapping $S:{\mathscr Y}\to {\mathscr X}$ such that $\langle
T(x),y\rangle=\langle x,S(y)\rangle$ for all $x\in {\mathscr X},
y\in {\mathscr Y}$. The unique mapping $S$ is denoted by $T^*$ and
is called the adjoint of $T$. It is easy to see that $T$ and $T^*$
must be bounded linear ${\mathscr A}$-module mappings. We denote by
${\mathcal L}({\mathscr X}, {\mathscr Y})$ the space of all
adjointable mappings from ${\mathscr X}$ to ${\mathscr Y}.$ We write
${\mathcal L}({\mathscr X})$ for the unital $C^*$-algebra ${\mathcal
L}({\mathscr X}, {\mathscr X})$; cf. \cite[p. 8]{LAN}. For every
$x\in {\mathscr X}$ the absolute value of $x$ is defined as the
unique positive square root of $\langle x,x \rangle ,$ that is,
$|x|=\langle x,x \rangle ^\frac{1}{2}$.

A Hilbert ${\mathscr A}$-module ${\mathscr X}$ can be embedded into
a certain $C^*$-algebra $\Lambda({\mathscr X})$. To see this, let us
denote by ${\mathscr F}={\mathscr X} \oplus {\mathscr A}$, the
direct sum of Hilbert ${\mathscr A}$-modules ${\mathscr X}$ and
${\mathscr A}$ equipped with the ${\mathscr A}$-inner product
$$\langle (x_1,a_1),(x_2,a_2)\rangle=\langle x_1,x_2\rangle+a^*_1a_2\,.$$ Identify each
$x\in {\mathscr X}$ with ${\mathscr A} \to {\mathscr X}, a \mapsto
xa$. The adjoint of this map is $x^*(y)=\langle x,y\rangle$. Set
$$\Lambda({\mathscr X})= \left\{ \left[ \begin{array}{cc}T&x\\y^*&a
\end{array}\right]: a\in {\mathscr A}, x,y\in {\mathscr X}, T\in {\mathcal L}({\mathscr X})\right\}.$$
$\Lambda({\mathscr X})$ is a $C^*$-subalgebra of ${\mathcal
L}({\mathscr F})$, called \emph{the linking algebra} of ${\mathscr
X}$. Then $${\mathscr X} \simeq \left[
\begin{array}{cc}0&{\mathscr X}\\0&0
\end{array}\right]\,, {\mathscr A} \simeq \left[ \begin{array}{cc}0&0\\0&{\mathscr
A}
\end{array}\right]\,, {\mathcal L}({\mathscr X}) \simeq \left[
\begin{array}{cc}{\mathcal L}({\mathscr X})&0\\0&0 \end{array}\right]\,.$$
Furthermore,$\langle x,y\rangle$ of ${\mathscr X}$ becomes the
product $x^*y$ in $\Lambda({\mathscr X})$ and the module
multiplication ${\mathscr X}\times {\mathscr A}\to {\mathscr X}$
becomes a part of the internal multiplication of $\Lambda({\mathscr
X})$.

We refer the reader to \cite{MUR} for undefined notions on
$C^*$-algebra theory and to \cite{FRA, LAN, R-W} for more
information on Hilbert $C^*$-modules.

A continuous real valued function $f$ on an interval $J$ is called
operator convex if for all $\lambda \in [0,1]$ and all self-adjoint
operators $A$ and $B$ acting on a Hilbert space $({\mathscr H},
\langle .,.\rangle)$, whose spectra are contained in
$J$,
$$f((1-\lambda)A+\lambda B) \leq (1-\lambda)f(A)+\lambda f(B),$$
where $\leq$ denotes the usual positive semi-definiteness. A
function $f: J \to {\mathbb R}$ is called operator concave if $-f$
is operator convex. A known operator Jensen equation says that if
$A$ is a self-adjoint operator with spectrum contained in an
interval $J$ on which $f$ is a convex function, then
\begin{eqnarray}\label{jen1}
f(\langle Ax, x\rangle) \leq \langle f(A)x,x\rangle
\end{eqnarray}
for every unit vector $x$; cf. \cite{M-P}.

\noindent By \emph{Hansen--Pedersen--Jensen's inequality} (see
\cite{F-F, H-P}) a function $f$ is operator convex (operator convex
and $f(0)\leq 0$, respectively) if and only if
\begin{eqnarray}\label{jen2}
f\left(\sum_{i=1}^nE_i^*A_iE_i\right) \leq
\sum_{i=1}^nE_i^*f(A_i)E_i
\end{eqnarray} for all self-adjoint bounded operators $A_i$ with
spectra contained in $J$ and all bounded operators $E_i$ with
$\sum_{i=1}^nE_i^*E_i=I$ ($\sum_{i=1}^nE_i^*E_i \leq I$,
respectively), where $I$ denotes the identity operator. The reader
is referred to \cite{FUR, P-F-H-S} for more information on operator
inequalities.


\section{A Hua type inequality in left Hilbert $C^*$-modules}

We start our work with the following generalized Hua inequality. In
our approach, we use neither convexity nor the classical Hua
inequality. Throughout this section, we assume that ${\mathscr X}$
and ${\mathscr Y}$ are Hilbert modules over a $C^*$-algebra
${\mathscr A}$.

\begin{theorem}\label{t1}
Let $f: [0,\infty) \to (0, \infty)$ be a function such that
$f(t)\geq t+M$ for some $M>0$. Then
\begin{eqnarray}\label{3.1}
|y\, (f(c)-c)^{-1/2} - x\, (f(c)-c)^{1/2}|^2 + c\,|x|^2 \geq c
f(c)^{-1}(f(c)-c)^{-1} |y|^2
\end{eqnarray}
for all positive central elements $c \in {\mathscr A}$ and all
elements $x,y\in {\mathscr X}$. The equality holds if and only if
$y=x\,f(c)$.
\end{theorem}
\begin{proof} By the functional calculus, $f(c)$ and $f(c)-c$ are invertible
positive elements of ${\mathscr A}$. Since $(f(c) - c)^{1/2}$ is a
central element,
\begin{eqnarray*}
& & |y\, (f(c) - c)^{-1/2} - x\, (f(c) - c)^{1/2}|^2 \\
&& \\
&& \qquad = (f(c) - c)^{-1/2}\,\langle y, y\rangle\,(f(c) -
c)^{-1/2}
- (f(c) - c)^{-1/2}\,\langle y, x\rangle\,(f(c) - c)^{1/2} \\
&& \\
&& \qquad \quad - (f(c) - c)^{1/2}\,\langle x, y\rangle\,(f(c) -
c)^{-1/2}
+ (f(c) - c)^{1/2}\,\langle x, x\rangle\,(f(c) - c)^{1/2} \\
&& \\
&& \qquad = (f(c) - c)^{-1}\,\langle y, y\rangle - \langle y,
x\rangle - \langle x, y\rangle + (f(c) - c)\,\langle x, x\rangle.
\end{eqnarray*}
Due to $(f(c)-c)^{-1} - f(c)^{-1} = cf(c)^{-1}(f(c)-c)^{-1}$, we
therefore get
\begin{eqnarray*}
&& |y\, (f(c) - c)^{-1/2} - x\, (f(c) - c)^{1/2}|^2 + c\,|x|^2 - c f(c)^{-1}(f(c) - c)^{-1} |y|^2 \\
&& \\
&& \qquad = f(c)^{-1}\,\langle y, y\rangle - \langle y, x\rangle - \langle x, y\rangle + f(c)\,\langle x, x\rangle \\
&& \\
&& \qquad = |y\, f(c)^{-1/2} - x\, f(c)^{1/2}|^2 \ \geq\ 0.
\end{eqnarray*}
Here the equality holds if and only if $y\, f(c)^{-1/2} = x\,
f(c)^{1/2}$, that is, $y = x\, f(c)$.
\end{proof}

\begin{example}
For Hilbert spaces ${\mathscr H}$ and ${\mathscr K}$, let ${\mathbb
B}({\mathscr H},{\mathscr K})$ denote the space of all bounded
linear operators from ${\mathscr H}$ to ${\mathscr K}$. Then
${\mathbb B}({\mathscr H},{\mathscr K})$ becomes a ${\mathbb
B}(\mathscr H)$-module by defining $\langle A,B\rangle:=A^*B$.
Suppose that $f: [0,\infty) \to (0, \infty)$ is a function such that
$f(t)\geq t+M$ for some $M>0$. Replacing $x, y$ in \eqref{3.1} by
$A, B$, respectively, and using this fact that the center of
${\mathbb B}(\mathscr H)$ is ${\mathbb C} I$ we get
\begin{eqnarray*}
\Big((f(c)-c)^{-1/2}\, B -
(f(c)-c)^{1/2}\,A\Big)^*\Big((f(c)-c)^{-1/2}\,
B-(f(c)-c)^{1/2}\,A\Big) + c\,A^*A\\
\geq c f(c)^{-1}(f(c)-c)^{-1}\,B^*B
\end{eqnarray*}
for all positive numbers $c \in [0,\infty)$ and all $A, B \in
{\mathbb B}({\mathscr H},{\mathscr K})$. The equality holds if and
only if $B=f(c)\, A$.
\end{example}

The following theorem is a norm extension of Hua's inequality.

\begin{theorem}\label{t2}
Let $f: [0,\infty) \to (0, \infty)$ be a function such that
$f(t)\geq t+M$ for some $M>0$. Then
\begin{eqnarray*}
\|y\, (f(c)-c)^{-1/2} - T(x)\, (f(c)-c)^{1/2}\|^2 +
\|c\|\|T\|^2\|x\|^2 \geq \| c f(c)^{-1}(f(c)-c)^{-1} |y|^2\,\|
\end{eqnarray*}
for all positive central elements $c \in {\mathscr A}$, all elements
$x\in {\mathscr X}$, $y\in {\mathscr Y}$ and all non-zero bounded
linear operators $T: {\mathscr X} \to {\mathscr Y}$.
\end{theorem}
\begin{proof}
Replacing $x$ by $T(x)$ in \eqref{3.1} we get
\begin{eqnarray*}
|y\, (f(c)-c)^{-1/2} - T(x)\, (f(c)-c)^{1/2}|^2 + c\,|Tx|^2 \geq c
f(c)^{-1}(f(c)-c)^{-1} |y|^2\,.
\end{eqnarray*}
utilizing the facts that $\|Tx\|/\|T\| \leq \|x\|$ and
$\|\,|z|\,\|^2=\|z\|^2\,\,(z\in{\mathscr A})$ we obtain
\begin{eqnarray*}
\|y\, (f(c)-c)^{-1/2} - T(x)\, (f(c)-c)^{1/2}\|^2 +
\|c\|\|T\|^2\|x\|^2 \geq \| c f(c)^{-1}(f(c)-c)^{-1} |y|^2\,\|\,.
\end{eqnarray*}
\end{proof}

Considering the elementary operator $T=u\otimes v$ defined for given
$u,v \in{\mathscr X}$ by $T(x)=u\langle v,x\rangle\,\,(x
\in{\mathscr X})$ and noting to the fact that $\|T\|=\|u\|\,\|v\|$
we get
\begin{corollary}
Let $f: [0,\infty) \to (0, \infty)$ be a function such that
$f(t)\geq t+M$ for some $M>0$. Then
\begin{eqnarray*}
\|y\, (f(c)-c)^{-1/2} - u\langle v,x\rangle\,(f(c)-c)^{1/2}\|^2 +
\|c\|\,\|u\|^2\|v\|^2\|x\|^2 \geq \| c
f(c)^{-1}(f(c)-c)^{-1}|y|^2\,\|
\end{eqnarray*}
for all positive central elements $c \in {\mathscr A}$, all elements
$x, y, u, v \in {\mathscr X}$.
\end{corollary}

If ${\mathscr X}$ and ${\mathscr Y}$ are assumed to be inner product
spaces ${\mathscr H}$ and ${\mathscr K}$, respectively, ${\mathscr
A} ={\mathbb C}$, $A\in {\mathbb B}({\mathscr H},{\mathscr K})$,
$f(t)=t+1$ and $c=\frac{\alpha}{\|A\|^2}$, then we deduce  the main
result of \cite{R-S} from Theorem \ref{t2} as follows.

\begin{corollary}\label{coro1}
Suppose that ${\mathscr H}$ and ${\mathscr K}$ are inner product
spaces, $A: {\mathscr H} \to {\mathscr K}$ is a bounded linear
operator and $\alpha>0$. Then
$$\|y-Ax\|^2 + \alpha \|x\|^2 \geq \frac{\alpha}{\|A\|^2+\alpha}
\|y\|^2$$ for all elements $x\in {\mathscr H}$ and $y\in {\mathscr
K}$.
\end{corollary}

\begin{remark}
If ${\mathscr H}$ is an inner product space, $w_i \in {\mathbb
C}\,\,(1 \leq i \leq n)$ and we consider the $n$-fold inner product
space ${\mathscr H}^n$ and $A(x_1, \cdots, x_n)=\sum_{i=1}^nw_ix_i$
in Corollary \ref{coro1} (see \cite{R-S}), then
$\|A\|^2=\sum_{i=1}^n|w_i|^2$ and so we get
$$\left\|y-\sum_{i=1}^n(w_ix_i)\right\|^2 + \alpha \sum_{i=1}^n\left(|w_i|^2\|x_i\|^2\right) \geq
\frac{\alpha}{\sum_{i=1}^n|w_i|^2+\alpha}\, \|y\|^2\,,$$ which is a
generalization of the main theorem of \cite{D-Y} (see also
\cite{DRA}). The case where ${\mathscr H}={\mathbb C}$ and
$w_i=1\,\,(1 \leq i\leq n)$ gives rise to the classical Hua's
inequality.

\end{remark}

\section{Hua's inequality for operator convex functions in Hilbert $C^*$-modules}

We first generalize operator Jensen inequality \eqref{jen1} in the
framework of Hilbert $C^*$-modules. In this section we assume that
${\mathscr X}$ is a Hilbert $C^*$-module over a unital $C^*$-algebra
${\mathscr A}$ with unit $e$.

\begin{theorem}
Let $f$ be an operator convex function on an interval $J$ containing
$0$, $f(0)\leq 0$ and let $T \in {\mathcal L}({\mathscr X})$ be
self-adjoint with spectrum contained in $J$. Then
\begin{eqnarray}\label{jen3}
f(\langle x, Tx\rangle) \leq \langle x,f(T)x\rangle
\end{eqnarray}
for every $x$ in the closed unit ball of ${\mathscr X}$.
\end{theorem}
\begin{proof}
To prove we utilize the linking algebra $\Lambda({\mathscr X})$ as a
$2\times 2$ matrix trick and the functional calculus for
self-adjoint elements of the $C^*$-algebra (see \cite{MUR}).
\begin{eqnarray*}
\left[
\begin{array}{cc}f(0)&0\\0&f(\langle x, Tx\rangle)
\end{array}\right] &=& f\left(\left[
\begin{array}{cc}0&0\\0&\langle x, Tx\rangle
\end{array}\right]\right)\\ &=& f\left(\left[
\begin{array}{cc}0&x\\0&0
\end{array}\right]^*\left[
\begin{array}{cc}T&0\\0&0
\end{array}\right]\left[
\begin{array}{cc}0&x\\0&0
\end{array}\right]\right)\\
&\leq& \left[
\begin{array}{cc}0&x\\0&0
\end{array}\right]^* f\left(\left[
\begin{array}{cc}T&0\\0&0
\end{array}\right]\right)\left[
\begin{array}{cc}0&x\\0&0
\end{array}\right]\\
&&\qquad\qquad \Big({\rm by~} \eqref{jen2} {\rm ~for~} n=1 {\rm
~and~} E_1= \left[
\begin{array}{cc}0&x\\0&0
\end{array}\right], \\
&& \qquad\qquad\qquad{\rm ~and~by~noting~to~} \|x\| \leq 1 \Leftrightarrow |x| \leq e\Big)\\
&=& \left[
\begin{array}{cc}0&x\\0&0
\end{array}\right]^*\left[
\begin{array}{cc}f(T)&0\\0& f(0)
\end{array}\right]\left[
\begin{array}{cc}0&x\\0&0
\end{array}\right]\\
&=& \left[
\begin{array}{cc}0&0\\0&\langle x, f(T)x\rangle
\end{array}\right]\,,
\end{eqnarray*}
whence we get \eqref{jen3}.
\end{proof}
Now we are ready to establish the second Hua's inequality in the
setting of Hilbert $C^*$-modules.

\begin{theorem}
Let $f$ be an operator convex function on an interval $J$ containing
$0$, $f(0)\leq 0$ and $S, R_i, T_i \in {\mathcal L}({\mathscr X})$
such that $S$ and $T_i$ are self-adjoint and the spectrum of
$S-\sum_{i=1}^n R_i^*T_iR_i$ and $T_i\,\,(1 \leq i \leq n)$ are
contained in $J$. Then for $x$ in the closed unit ball of ${\mathscr
X}$,
\begin{eqnarray*}
\left\langle x, \left[f\left(S-\sum_{i=1}^n R_i^*T_iR_i\right)+
\sum_{i=1}^n R_i^*f(T_i)R_i\right] x \right\rangle\geq
k_n^{-1}f\left(k_n\left\langle x, Sx\right\rangle\right)\,,
\end{eqnarray*}
where $k_n=\left(1+\sum_{i=1}^n\|R_ix\|^2\right)^{-1}$.
\end{theorem}
\begin{proof}
\begin{align*}
&\left\langle x, \left[f\left(S-\sum_{i=1}^n R_i^*T_iR_i\right)+
\sum_{i=1}^n R_i^*f(T_i)R_i\right] x \right\rangle\\
&=\left\langle x, f\left(S-\sum_{i=1}^n
R_i^*T_iR_i\right)x\right\rangle+\left\langle x,
\sum_{i=1}^n R_i^*f(T_i)R_i x \right\rangle\\
&= \left\langle x, f\left(S-\sum_{i=1}^n R_i^*T_iR_i\right)x
\right\rangle+ \sum_{i=1}^n \|R_ix\|^2\left\langle
\frac{R_ix}{\|R_ix\|}, f(T_i)\frac{R_ix}{\|R_ix\|}\right\rangle\\
&\geq f\left(\left\langle x, \left(S-\sum_{i=1}^n
R_i^*T_iR_i\right)x \right\rangle\right)+ \sum_{i=1}^n
\|R_ix\|^2f\left(\left\langle \frac{R_ix}{\|R_ix\|},
T_i\left(\frac{R_ix}{\|R_ix\|}\right)\right\rangle\right)\quad ({\rm
by~} \eqref{jen3})\\
&= 1 \cdot f\left(\left\langle x, \left(S-\sum_{i=1}^n
R_i^*T_iR_i\right)x \right\rangle\right)+ \sum_{i=1}^n
\|R_ix\|^2f\left(\left\langle x,
\frac{R_i^*T_iR_ix}{\|R_ix\|^2}\right\rangle\right)\\
&\geq k_n^{-1}f\left(k_n \left\langle x, \left(S-\sum_{i=1}^n
R_i^*T_iR_i\right)x \right\rangle + k_n\sum_{i=1}^n \|R_ix\|^2
\left\langle x,
\frac{R_i^*T_iR_ix}{\|R_ix\|^2}\right\rangle\right)\\
&\qquad\qquad\qquad\qquad\qquad\qquad\qquad\qquad\qquad\qquad({\rm
by~the~classical~weighted~Jensen~inequality})\\
&\geq k_n^{-1}f\left(k_n\left\langle x,\left(S-\sum_{i=1}^n
R_i^*T_iR_i + \sum_{i=1}^n \|R_ix\|^2
\frac{R_i^*T_iR_i}{\|R_ix\|^2}\right)x\right\rangle\right)\\
&= k_n^{-1}f\left(k_n\left\langle x, Sx\right\rangle\right)\,,
\end{align*}
where $k_n=\left(1+\sum_{i=1}^n\|R_ix\|^2\right)^{-1}$.
\end{proof}
As a special case in which ${\mathscr X}$ is ${\mathbb C}$ as a
${\mathbb C}$-module, $f(t)=t^2$, $x=1$, and $S$, $T_i\,\,(1 \leq i
\leq n)$ and $R_i\,\,(1 \leq i \leq n)$ are given real numbers
$\delta$, $\alpha x_i\,\,(1 \leq i \leq n)$ and $\alpha^{-1/2}$,
respectively, we get the classical Hua's inequality.


\section{An operator Hua's inequality}

In this section we present an operator Hua's inequality and show
that it is equivalent to operator convexity.

\begin{theorem}\label{t3}
Let ${\mathscr H}$ be a Hilbert space, let $f$ be an operator convex
function on an interval $J$ and let $C_i\,\,(1 \leq i \leq n)$ be
arbitrary operators and $B, A_i\,\,(1\leq i\leq n)$ be self-adjoint
operators operators such that the spectra of $B-\sum_{i=1}^n
C_i^*A_iC_i$ and $A_i$ are contained in $J$. Then
\begin{eqnarray}\label{5.1}
f\left(B-\sum_{i=1}^n C_i^*A_iC_i\right) + \sum_{i=1}^n
C_i^*f(A_i)C_i \geq D^{-1}f(DBD)D^{-1}\,,
\end{eqnarray}
where $D=\left(I+\sum_{i=1}^nC_i^*C_i\right)^{-1/2}$.
\end{theorem}
\begin{proof} It follows from
$$DD+\sum_{i=1}^nDC_i^*C_iD=D\left(I+\sum_{i=1}^nC_i^*C_i\right)D=I$$
and operator convexity of $f$ that
\begin{eqnarray*}
Df\left(B-\sum_{i=1}^n C_i^*A_iC_i\right)D + \sum_{i=1}^n
DC_i^*f(A_i)C_iD &\geq& f\Big(DBD-\sum_{i=1}^nDC_i^*A_iC_iD\\
&&+\sum_{i=1}^nDC_i^*A_iC_iD\Big)\\
&=&f(DBD)\,.
\end{eqnarray*}
Due to the fact that $T\geq S$ implies $R^*TR \geq R^*SR$, we deduce
\eqref{5.1}.
\end{proof}

\begin{remark}
If we use a concave operator function, then the inequality in
\eqref{5.1} will be reversed.
\end{remark}

Applying Theorem \ref{t3} to the convex functions $f_1(t)=t^{-1}$
and $f_2(t)=t^p\,\,(1\leq p\leq 2 {\rm ~or~} -1 \leq p \leq 0)$, and
concave functions $f_3(t)=t^p\,\,(0\leq p\leq 1)$ and $f_4(t)=\log
t$ on $(0, \infty)$ we get inequalities (i)-(iv) of the following
Corollary, respectively. The inequalities of (i) and (iv) are
operator extensions of main result of \cite{WAN}.
\begin{corollary}
Let ${\mathscr H}$ be a Hilbert space and let $C_i\,\,(1 \leq i \leq
n)$ be arbitrary operators and $B$ and $A_i\,\,(1\leq i\leq n)$ be
self-adjoint operators acting on ${\mathscr H}$ such that
$B-\sum_{i=1}^n C_i^*A_iC_i>0$ and $A_i>0$. Then
\begin{align*}
& (i)~ \left(B-\sum_{i=1}^n C_i^*A_iC_i\right)^{-1} + \sum_{i=1}^n
C_i^*A_i^{-1}C_i
\geq D^{-1}(DBD)^{-1}D^{-1}\,;\\
& (ii)~ \left(B-\sum_{i=1}^n C_i^*A_iC_i\right)^{p} + \sum_{i=1}^n
C_i^*A_i^{p}C_i
\geq D^{-1}(DBD)^{p}D^{-1}\quad (p \in [-1,0]\cup[1,2])\,;\\
& (iii)~ \left(B-\sum_{i=1}^n C_i^*A_iC_i\right)^{p} + \sum_{i=1}^n
C_i^*A_i^{p}C_i
\leq D^{-1}(DBD)^{p}D^{-1}\quad(p\in[0,1])\,;\\
& (iv)~ \log\left(B-\sum_{i=1}^n C_i^*A_iC_i\right) + \sum_{i=1}^n
C_i^*\log (A_i)C_i \leq D^{-1}\log(DBD)D^{-1}\,,
\end{align*}
where $D=(I+\sum_{i=1}^nC_i^*C_i)^{-1/2}$.
\end{corollary}

If ${\mathscr H}$ is of dimension one, $\delta, x_1, \cdots, x_n \in
{\mathbb R}$, $\alpha >0$ and we take $C_i=\alpha^{-1/2}\,\,(1 \leq
i \leq n)$, $B=\delta$ and $A_i=\alpha x_i\,\,(1 \leq i\leq n)$ in
Theorem \ref{t3}, then
$D=\left(\frac{\alpha}{n+\alpha}\right)^{1/2}$ we obtain the main
theorem of \cite{P-P} as follows.

\begin{corollary} \cite[Theorem 1]{P-P}
Let $f: J \to {\mathbb R}$ be a convex function and let $\alpha,
\delta, x_1, \cdots, x_n$ be real numbers such $\alpha>0$ and
$\delta-\sum_{i=1}^n x_i, \alpha x_1\cdots, \alpha x_n \in J$. Then
\begin{eqnarray}\label{pec}
f\left(\delta-\sum_{i=1}^nx_i\right) +
\sum_{i=1}^n\alpha^{-1}f\left(\alpha x_i\right)\geq
\frac{\alpha+n}{\alpha}f\left(\frac{\alpha
\delta}{\alpha+n}\right)\,.
\end{eqnarray}
\end{corollary}

Now, we shall show that our operator Hua's inequality \eqref{5.1}
implies the Hansen--Pedersen--Jensen operator inequality in the case
where the value of the real function $f$ at $0$ is non-positive.

\begin{theorem}\label{t4}
Let $J$ be an interval containing $0$ and let $f$ be a continuous
real valued function defined on $J$ with $f(0)\leq 0$. Let operator
Hua's inequality \eqref{5.1} hold, where $C_i\,\,(1 \leq i \leq n)$
are arbitrary operators and $B, A_i\,\,(1\leq i\leq n)$ are
self-adjoint operators such that the spectra of $B-\sum_{i=1}^n
C_i^*A_iC_i$ and $A_i$ are contained in $J$. Then $f$ is operator
convex.
\end{theorem}
\begin{proof}
Let $X_i\,\,(1 \leq i \leq n)$ be self-adjoint operators on a
Hilbert space ${\mathscr H}$ with spectra contained in $J$ and $E_i$
be arbitrary operators such that $\sum_{i=1}^nE_i^*E_i < I$, where
$<$ denotes the strict positivity. Recall that by a strict positive
operator we mean an invertible positive one.

\noindent Replacing $B$ by $\sum_{i=1}^nC_i^*A_iC_i$ in \eqref{5.1}
we get
$$\sum_{i=1}^n C_i^*f(A_i)C_i \geq
f(0)+\sum_{i=1}^n C_i^*f(A_i)C_i \geq
D^{-1}f\left(D\sum_{i=1}^nC_i^*A_iC_iD\right)D^{-1}\,,$$ that is
\begin{eqnarray}\label{5.3}
\sum_{i=1}^n (C_iD)^*f(A_i)C_iD \geq
f\left(\sum_{i=1}^n(C_iD)^*A_i(C_iD)\right)\,. \end{eqnarray} Set
$C_i=E_i\left(I-\sum_{i=1}^nE_i^*E_i\right)^{-1/2}$. Then
$$\sum_{i=1}^n C_i^*C_i+I = \sum_{i=1}^n E_i^*E_i \left(I-\sum_{i=1}^n
E_i^*E_i\right)^{-1}+I=\left(I-\sum_{i=1}^n
E_i^*E_i\right)^{-1}\,,$$ whence
$$D=\left(\sum_{i=1}^n C_i^*C_i+I\right)^{-1/2}= \left(I-\sum_{i=1}^n
E_i^*E_i\right)^{1/2}$$ so $C_iD=E_i$. It follows from \eqref{5.3}
that
$$\sum_{i=1}^n E_i^*f(X_i)E_i \geq
f\left(\sum_{i=1}^nE_i^*X_iE_i\right)\,.$$ By using the same method
of the proof of (iv) $\Rightarrow$ (i) in \cite[Theorem
1.10]{P-F-H-S} with the matrix $P=\left[
\begin{array}{cc}1&0\\0&0
\end{array}\right]$ replaced by the diagonal matrix
$P=\left[
\begin{array}{cc}1-\varepsilon&0\\ 0&\varepsilon
\end{array}\right]$ with $0<\varepsilon <1$ and then applying the continuity of $f$ as $\varepsilon \to
0$, we deduce that the latter inequality holds for all arbitrary
operators $E_i$ with $\sum_{i=1}^nE_i^*E_i \leq I$ and the function
$f$ is therefore operator convex.
\end{proof}

Thus we conclude that operator Hua's inequality \eqref{5.1} is
equivalent to the Hansen--Pedersen--Jensen operator inequality. One
can similarly deduce that a continuous real function satisfying
\eqref{pec} is convex.\\


\textbf{Acknowledgement.} The author would like to sincerely thank
Professor Tsuyoshi Ando for their very useful comments and to
express his gratitude to Professor Frank Hansen and Professor Yuki
Seo for their helpful suggestions on the last paragraph of the proof
of Theorem \ref{t4}.


\end{document}